\documentclass[a4paper,12pt,reqno]{amsart}

\usepackage[T1]{fontenc}
\usepackage[utf8]{inputenc}
\usepackage{lmodern}
\usepackage[english]{babel}

\usepackage[%
    a4paper,
	left=2.5cm,       
	right=2.5cm,      
	top=3.5cm,        
	bottom=3.5cm,     
	heightrounded,    
	bindingoffset=0mm 
]{geometry}

\usepackage{amsmath}
\usepackage{amssymb}
\usepackage{mathtools}
\usepackage{mathrsfs}
\usepackage[dvipsnames]{xcolor}
\usepackage{graphicx}
\usepackage{enumitem}
\usepackage{hyperref} 

\usepackage[noabbrev,capitalize]{cleveref}
\usepackage{bookmark}
\usepackage[initials]{amsrefs}

\numberwithin{equation}{section}

\renewcommand{\phi}{\varphi}

\theoremstyle{definition}
\newtheorem{Def}{Definition}[section]

\theoremstyle{plain}
\newtheorem{prop}[Def]{Proposition}

\newtheorem{thm}[Def]{Theorem}
\newtheorem{lem}[Def]{Lemma}

\theoremstyle{remark}
\newtheorem{rem}[Def]{Remark}

\DeclarePairedDelimiter{\abs}{|}{|}
\DeclarePairedDelimiter{\norm}{\|}{\|}
\DeclarePairedDelimiter{\set}{\{}{\}}

\DeclareMathOperator{\divergence}{div}

\newcommand{\e}{\varepsilon}
\newcommand{\R}{\mathbb{R}}
\newcommand{\Lip}{\text{Lip}}
\newcommand{\rel}{{\rm rel}}
\DeclareMathOperator{\BMO}{BMO}

\allowdisplaybreaks

\begin{document}

\title[Uniqueness and vanishing viscosity for Euler]{Weak-strong uniqueness and vanishing viscosity for incompressible Euler equations in exponential spaces}

\author[L.~De Rosa]{Luigi De Rosa}
\address[L.~De Rosa]{Department Mathematik und Informatik, Universität Basel, Spiegelgasse~1, 4051 Basel, Switzerland}
\email{luigi.derosa@unibas.ch}

\author[M.~Inversi]{Marco Inversi}
\address[M.~Inversi]{Department Mathematik und Informatik, Universität Basel, Spiegelgasse~1, 4051 Basel, Switzerland}
\email{marco.inversi@unibas.ch}

\author[G.~Stefani]{Giorgio Stefani}
\address[G.~Stefani]{Scuola Internazionale Superiore di Studi Avanzati (SISSA), via Bonomea~265, 34136 Trieste (TS), Italy}
\email{giorgio.stefani.math@gmail.com}

\date{\today}



\thanks{\textit{Acknowledgments}.
The authors thank Stefano Spirito and Emil Wiedemann for many precious comments on a preliminary version of this paper. 
The first two authors are partially supported by the 2015 ERC Grant 676675 FLIRT--Fluid Flows and Irregular Transport. The third author is a member of Istituto Nazionale di Alta Matematica (INdAM), Gruppo Nazionale per l'Analisi Matematica, la Probabilità e le loro Applicazioni (GNAMPA), is partially supported by the INdAM--GNAMPA 2022 Project \textit{Analisi geometrica in strutture subriemanniane}, codice CUP\_E55\-F22\-000\-270\-001, by the INdAM--GNAMPA 2023 Project \textit{Problemi variazionali per funzionali e operatori non-locali}, codice CUP\_E53\-C22\-001\-930\-001, and has received funding from the European Research Council (ERC) under the European Union’s Horizon 2020 research and innovation program (grant agreement No.~945655).
}

\begin{abstract}
In the class of admissible weak solutions, we prove a weak-strong uniqueness result for the incompressible Euler equations assuming that the symmetric part of the gradient belongs to $L^1_{\rm loc}([0,+\infty);L^{\exp}(\R^d;\R^{d\times d}))$, where $L^{\exp}$ denotes the Orlicz space of exponentially integrable functions. Moreover, under the same assumptions on the limit solution to the Euler system, we obtain the convergence of vanishing-viscosity Leray--Hopf weak solutions of the Navier--Stokes equations.
\end{abstract}

\maketitle

\section{Introduction}

The present work is devoted to the analysis of the Euler and the Navier--Stokes equations in the context of incompressible fluids. Despite their importance in modeling several natural phenomena, their rigorous mathematical study remains vastly incomplete. Indeed, even though these equations were proposed hundreds of years ago, major questions such as existence and smoothness of solutions presently remain extremely challenging open problems.   

In this paper, we focus on the uniqueness of solutions of the Euler equations and on the inviscid limit of solutions of the Navier--Stokes equations. Our results should be  compared with the ones obtained in~\cite{MR08}.
In fact, our approach allows us not only to upgrade the uniqueness result proved in~\cite{MR08} to a \emph{weak-strong} one, but also to weaken the hypothesis on the spatial regularity of the gradient of the velocity, actually noticing that such an assumption is needed on its symmetric part only.

In the next part of this introduction, we summarize our main results together with a brief overview of the (at the best of the authors' knowledge) sharpest  achievements already known in the literature on this topic. We will state our results in the whole space~$\R^d$ but we remark that the very same proofs apply on the $d$-dimensional torus~$\mathbb{T}^d$ with straightforward modifications. 

\subsection{Euler equations}
For a given time $T\in(0,+\infty)$, the \emph{incompressible Euler equations} are the following system of partial differential equations
\begin{equation}
\label{euler}
\tag{$\mathbf{E}$}
    \begin{cases}
    \partial_t u + \divergence{(u \otimes u)} + \nabla p = 0 & \text{in }  \R^d\times (0,T),
    \\ \divergence{u} = 0 &\text{in }  \R^d\times (0,T),
    \\ u(\,\cdot\,,0) = u_0 &\text{in } \R^d,
    \end{cases} 
\end{equation}
where $u\colon\R^d\times [0,T]\to\R^d$ represents the velocity of the fluid, $p\colon \R^d\times [0,T]\to\R$ its hydrodynamic pressure and $u_0\colon\R^d\to \R^d$ is any given initial divergence-free vector field.

We recall the standard notion of weak solution of the system~\eqref{euler}, which can be formally obtained by multiplying the first two equations of~\eqref{euler} by two (regular enough) test functions~$\varphi$ and~$q$ and integrating by parts. 

\begin{Def}[Weak solution] 
\label{euler system}
Let $u_0 \in L^2(\R^d; \R^d)$ be such that $\divergence{u_0}=0$.
We say that a function $u \in L^\infty([0, T]); L^2(\R^d; \R^d))$ is a \emph{weak solution} of the incompressible Euler equations~\eqref{euler} with initial datum $u_0$ if   
\begin{equation*}
    \int_0^\tau \int_{\R^d} \left(u \cdot \partial_t \phi+  u\otimes u \colon \nabla  \phi \right)\,dx \, dt =  \int_{\R^d} u(x, \tau) \cdot \phi(x, \tau)\, dx - \int_{\R^d} u_0(x) \cdot \phi(x, 0)\, dx
\end{equation*}
and
\begin{equation*}
    \int_{\R^d} u(x,\tau) \cdot \nabla q(x) \, dx = 0
\end{equation*}
for a.e.\ $\tau \in (0, T)$, whenever  
\begin{equation*}
\phi \in \text{Lip}(\R^d\times (0,T);\R^d)\cap \text{Lip}((0,T);L^2(\R^d;\R^d)\cap W^{1,\infty}(\R^d;\R^d)),
\qquad
\divergence{\phi}=0,
\end{equation*}
and $q \in W^{1,2}(\R^d)$.  
\end{Def} 

Note that the regularity assumptions on the test functions~$\varphi$ and~$q$ guarantee that all the terms involved in the weak formulation are well defined. We also remark that, by standard density arguments, the definition we used readily follows from the usual definition in which $\varphi$ and $q$ are assumed to be smooth. 

The existence of weak solutions of~\eqref{euler} has been proved in~\cite{W11}, while their uniqueness is known to fail~\cites{Sch93,Is18}, even from every smooth initial datum~\cite{DH21}. 
Non-uniqueness results in the presence of an external forcing term have been proved by Vishik in the seminal papers~\cites{V1,V2} (also see~\cite{ABCDGJK22}).

\begin{rem}[Weak solutions are $C({[}0,T{]};w-L^2)$]
\label{rem:continuity_w_L2}
Every weak solution 
\begin{equation*}
u \in L^\infty([0,T];L^2(\R^d;\R^d))
\end{equation*}
as in \cref{euler system} can be redefined on a negligible set of times in order to guarantee that 
\begin{equation*}
u \in C([0,T];w-L^2(\R^d;\R^d)),
\end{equation*}
that is, $u$ is continuous in time with values in $L^2$ endowed with the weak topology.
For a proof of this statement, we refer the reader to~\cite{DS10}*{Appendix~A} for example.  We also remark that the same result holds true for weak solutions of the Navier--Stokes equations~\eqref{navierstokes} as defined in \cref{navierstokes system} below.
\end{rem}

Throughout this note, we will always deal with the $L^2$ weakly continuous representative of the weak solution under consideration. This will improve the readability of some statements.  

\subsection{Admissible weak solutions}
Since the notion of weak solution given above in \cref{euler system} is far from giving the well-posedness of the Euler system~\eqref{euler}, one usually restricts the class of weak solutions by imposing some additional constraints. The most frequently considered is to assume that, at a.e.\ time, the kinetic energy of the weak solution is below the one of the initial datum.

\begin{Def}[Admissible weak solution] 
\label{admissible solution}
Let $u \in L^\infty([0,T];L^2(\R^d;\R^d))$ be a weak solution of the Euler equations~\eqref{euler} with initial datum $u_0 \in L^2(\R^d;\R^d)$, $\divergence u_0=0$. We say that $u$ is \emph{admissible} if 
\begin{equation}
     \int_{\R^d} \abs{u(x,t)}^2 \, dx \leq \int_{\R^d} \abs{u_0(x)}^2 \, dx
\label{admissible energy inequality}
\end{equation}
for a.e.\  $t \in (0, T)$.
\end{Def}

In virtue of \cref{rem:continuity_w_L2} and by lower semicontinuity of the norm with respect to the weak convergence, up to possibly redefining the weak solution~$u$ on a negligible set of times, the energy inequality in~\eqref{admissible energy inequality} is actually valid for every time $t>0$.\\

The notion of admissible weak solution given in \cref{admissible solution} above plays a very important physical role, since it prohibits solutions to generate kinetic energy from nowhere. However, the lack of compactness for the Euler system~\eqref{euler} (due to the absence of viscosity) does not allow to prove the existence of admissible weak solutions starting from a general divergence-free initial datum $u_0\in L^2(\R^d;\R^d)$. Moreover, in virtue of~\cites{BDSV19,DRS21}, uniqueness still fails in the class of admissible weak solutions. 

\subsection{Weak-strong uniqueness for admissible weak solutions}
The lack of uniqueness for admissible weak solutions of the Euler system~\eqref{euler} motivates the current research on \emph{weak-strong} uniqueness results, that is, two admissible weak solutions of~\eqref{euler} starting from the same divergence-free initial datum must coincide if at least one of them satisfies suitable additional regularity assumptions.
It is known that, in the class of admissible weak solutions, it is sufficient to assume that one solution satisfies $u\in C^1_{x,t}$ with symmetric gradient such that $\nabla^s u\in L^1_t(L^\infty_x)$. We refer the interested reader to~\cite{W18} for a proof of this statement, together with a more detailed overview of the classical results in this direction.

Building on a logarithmic interpolation inequality between the (real) \emph{Hardy space} $\mathcal{H}^1(\R^d)$ and the \emph{logarithmic Zygmund space} $L \log L(\R^d)$, in~\cite{MR08}*{Theorem~1.2} the authors prove that uniqueness of weak solutions of the Euler system~\eqref{euler} still holds if, for some $\sigma>0$, $u\in L^\infty_t(L^{2+\sigma}_x)$ and $\nabla u\in L^1_t(\BMO_x)$, thus relaxing to $\BMO$ the standard $L^\infty$ assumption on the spatial regularity of the velocity gradient. 
Motivated by the uniqueness result of~\cite{MR08} and having in mind the results achieved in~\cite{BDS11} in the more general class of \emph{measured-valued} weak solutions of the Euler system~\eqref{euler}, our first result upgrades~\cite{MR08}*{Theorem~1.2} to a weak-strong uniqueness result under the weaker assumption that $\nabla^s U\in L^1_t(L^{\exp}_x)$.
As a matter of fact, we prove the same logarithmic inequality of~\cite{MR08}*{Theorem~1.1} in this more general framework. Here $L^{\exp}$ denotes the Orlicz space of exponentially integrable functions as defined in \cref{sec:prel} below.

In order to state our first main result, we introduce the space of admissible weak solutions we will employ throughout our paper.

\begin{Def}[The space $X_{u_0}$]
\label{X-u_0}
Let $u_0 \in L^2(\R^d;\R^d)$ be such that $\divergence{u_0}=0$.
We say that $u$ belongs to the space $X_{u_0}$ if $u \in L^\infty([0,T];L^2(\R^d;\R^d))$ is an admissible weak solution to~\eqref{euler} and there exists $\sigma>0$ such that $u \in L^\infty([0,T];L^{2+\sigma}(\R^d;\R^d))$.
\end{Def}

We are now ready to state our first main result concerning weak-strong uniqueness of admissible solutions of the Euler system~\eqref{euler}.

\begin{thm}[Weak-strong uniqueness] 
\label{t:main1}
Let $u_0 \in L^2(\R^d;\R^d)$ be such that $\divergence{u_0}=0$ and let $u, U \in X_{u_0}$.
If
\begin{equation*}
\nabla^s U \in L^1([0,T];L^{\exp}(\R^d;\R^d)),
\end{equation*}
then $u = U$.
\end{thm}

The proof of \cref{t:main1}, similarly to that of~\cite{MR08}*{Theorem 1.2}, is based on the \emph{relative energy method}, that is, we study the behavior of the \emph{relative energy} 
\begin{equation}\label{en_rel_ineq}
   E_{\rel}(\tau) = \frac{1}{2} \int_{\R^d} \abs{u(x,\tau) - U(x, \tau)}^2 \, dx,
\qquad
\tau\ge0,
\end{equation}
where $u,U\in X_{u_0}$ are as in the statement of \cref{t:main1}. 
The key observation of this method is that, if~$U$ can be used as a test function in the weak formulation of~$u$, then we find that
\begin{equation}
\label{dis_bella}
 E_{\rel}(\tau) \leq \int_0^\tau \int_{\R^d} \nabla^s U \colon (U-u) \otimes (u-U) \,dx \,dt
\end{equation} 
for $\tau>0$.

Once inequality~\eqref{dis_bella} is established, the fact that $E_{\rel} \equiv 0$ then follows via a standard Gr\"onwall-type argument. 
For a more detailed account on this approach, we refer the interested reader to~\cite{W18} and to the references therein.

It is worth noticing that the usual proof of  inequality~\eqref{en_rel_ineq} requires suitable regularity assumptions on $U$ and, in the standard setting, one typically assumes that $U\in C^1_{x,t}$ as already mentioned before, see~\cite{W18}*{Theorem~11.1}.

In our setting, $U$ may not satisfy such a strong regularity.
To overcome this issue, in the inequality~\eqref{dis_bella} we replace $U$ by its space regularization~$U_{\varepsilon}$, with~$\e>0$.
However, due to the non-linearity of the Euler equations~\eqref{euler}, this replacement inevitably generates some error terms that must be taken under control. At this point, with some careful manipulations, we prove that an exponential integrability assumption only on the symmetric part of the gradient $\nabla^s U$ is enough to guarantee that all the error terms vanish in the limit as $\e\to 0^+$. It is worth to mention that, to get uniqueness, the hypothesis on the only symmetrical part of the gradient also appears in the compressible context \cite{FGJ19}.  
Let us also mention that adding an external forcing term $F$ does not affect the validity of \cref{t:main1}, as well as of \cref{weak-strong uniqueness bis} below, provided that one assumes the right integrability. 
See \cref{R:ext_force} for a more extensive discussion. 

As a by-product of our approach, we also give a simpler proof of~\cite{BDS11}*{Theorem~2} when the class of measure-valued admissible weak solutions is upgraded to the usual notion of admissible weak solutions as in \cref{euler system}. In this case, one does not even need to consider solutions with the spatial integrability $L^{2+\sigma}$ for some $\sigma>0$, since the standard notion of finite energy solutions is enough.

\begin{thm}[Standard weak-strong uniqueness]
\label{weak-strong uniqueness bis}
Let $u_0 \in L^2(\R^d;\R^d)$ be such that $\divergence u_0=0$ and let $u, U$ be two admissible weak solutions with initial condition $u_0$.
If 
\begin{equation*}
\nabla^s U \in L^1([0,T];L^\infty(\R^d;\R^{d\times d})),
\end{equation*}
then $u=U$. 
\end{thm}

In \cref{t:main1} and \cref{weak-strong uniqueness bis} above we assumed the solution $U$ to be admissible. In fact, under the respective  assumptions on the symmetric part of the gradient of the weak solution, one can prove the conservation of the energy, as stated in the following result. 

\begin{prop}[Energy conservation] \label{energy conservation for U}
Let $U \in L^\infty([0,T];L^2(\R^d;\R^d))$ be a weak solution of~\eqref{euler}.
If either
\begin{equation}
\label{energy cons exp}
    U\in L^\infty([0,T];L^{2+\sigma}(\R^d;\R^d))\ \text{for some}\  \sigma>0,\ 
    \nabla^s U \in L^1([0,T];L^{\exp}(\R^d;\R^{d\times d})),
\end{equation}
or
\begin{equation}
\label{energy cons standard}
U \in L^\infty([0, T]; L^2(\R^d; \R^d)),\ \nabla^s U \in L^1([0, T]; L^\infty(\R^d; \R^{d\times d})),
\end{equation}
then $U$ preserves the energy. 
\end{prop}

Since classical, although not present in the literature (at the best of our knowledge), we will omit the proof of \cref{weak-strong uniqueness bis}, as well as of the energy conservation in \cref{energy conservation for U}, under the assumptions~\eqref{energy cons standard}.
Anyway, the interested reader can easily reconstruct the proof of these two results by adapting the ones we will give for Theorem \ref{t:main1} and the energy conservation assuming \eqref{energy cons exp} with minor modifications. We also remark that the spatial regularity of the  $\nabla^s U$ in the energy conservation \eqref{energy cons exp} is essentially sharp in view of \cite{CL21}*{Theorem~1.11} in which non-conservative solutions $U\in L^1_t(C^{1-}_x)$ have been constructed.

\subsection{Leray--Hopf weak solutions of Navier--Stokes equations} 

In the presence of a positive \emph{viscosity} $\nu>0$, the Euler equations~\eqref{euler} generalize to the \emph{incompressible Navier--Stokes equations}
\begin{equation}
\label{navierstokes}
\tag{$\mathbf{NS}$}
    \begin{cases}
    \partial_t u + \divergence{(u \otimes u)} + \nabla p -\nu\Delta u= 0 & \text{in }  \R^d\times (0,T),
    \\ \divergence{u} = 0 &\text{in }  \R^d\times (0,T),
    \\ u(\,\cdot\,,0) = u_0 &\text{in } \R^d.
    \end{cases} 
\end{equation}
As for the Euler equations~\eqref{euler}, we are interested in suitably defined weak solutions of the Navier--Stokes system~\eqref{navierstokes}, known as \emph{Leray--Hopf weak solutions}. 

\begin{Def}[Leray--Hopf weak solution] 
\label{navierstokes system}
Let $u_0 \in L^2(\R^d;\R^d)$ be such that $\divergence{u_0}=0$ and let $\nu>0$. We say that a function $u \in L^\infty([0, T]); L^2(\R^d;\R^d))\cap L^2([0,T];W^{1, 2}(\R^d;\R^d))$ is a \emph{Leray--Hopf weak solution} to \eqref{navierstokes} with initial datum $u_0$ if   
\begin{align*}
    \int_{\tau_1}^{\tau_2} \int_{\R^d} \big(u \cdot \partial_t \phi
    &
    +  u\otimes u \colon \nabla  \phi -\nu \nabla u :\nabla \varphi\big) \,dx \, dt \\
    & =  \int_{\R^d} u(x, \tau_2) \cdot \phi(x, \tau_2)\, dx
    - \int_{\R^d} u(x,\tau_1) \cdot \phi(x, \tau_1)\, dx
\end{align*}
and
\begin{equation*}
    \int_{\R^d} u(x,\tau) \cdot \nabla q(x) \, dx = 0
\end{equation*}
for a.e.\ $\tau,\tau_1,\tau_2 \in (0, T)$, including $\tau_1=0$ and with $\tau_1< \tau_2$, whenever  
\begin{equation*}
\phi \in \text{Lip}(\R^d\times (0,T);\R^d)\cap \text{Lip}((0,T);  W^{1,\infty}(\R^d;\R^d)\cap W^{1,2}(\R^d;\R^d)),
\qquad
\divergence{\phi}=0,
\end{equation*}
and $q \in W^{1,2}(\R^d)$.  Moreover, the following \emph{energy inequality} holds
\begin{equation}
    \frac{1}{2}\int_{\R^d}|u(x,\tau_2)|^2\,dx+\nu \int_{\tau_1}^{\tau_2}\int_{\R^d} |\nabla u(x,t)|^2\,dx \,dt\leq \frac{1}{2}\int_{\R^d}|u(x,\tau_1)|^2\,dx,
\end{equation}
for every $\tau_2\in (0,T)$ and for a.e.\ $\tau_1<\tau_2$ including $\tau_1=0$.
\end{Def} 

As already observed for \cref{euler system}, the regularity of the test functions in \cref{navierstokes system} is not the classical one in which one assumes the test functions~$\varphi$ and~$q$ to be smooth. However, \cref{navierstokes system} can be derived by standard density arguments from the one employing smooth test functions.
We refer the reader to~\cite{RRS16} for an account on this subject, as well as for a panoramic on the different notions of weak solutions for the Navier--Stokes system~\eqref{navierstokes}.

The existence of Leray--Hopf weak solutions, as well as their regularity properties with respect to the time variable, have been proved in the fundamental works of Leray~\cite{L34} and Hopf~\cite{H51} in the physically relevant cases $d=2,3$. 
We refer the interested reader to \cite{Lions96}*{Section~3.1} for the corresponding results in arbitrary dimension.

On the other hand, uniqueness of Leray--Hopf weak solutions remains a formidable open problem. 
Very recently, non-uniqueness in the presence of an external force has been established in~\cite{ABC21}.

\subsection{Vanishing-viscosity limit of Leray--Hopf weak solutions}

To date, in addition to the uniqueness problem, the \emph{vanishing-viscosity limit} of Leray--Hopf weak solutions represents a question of major importance, that is, whether or not Leray--Hopf weak solutions $(u^\nu)_{\nu>0}$ of the Navier--Stokes system~\eqref{navierstokes} converge to a weak solution of the Euler system~\eqref{euler} in the limit as the viscosity vanishes, $\nu\to 0^+$. 

The \emph{inviscid limit problem} for Leray--Hopf weak solutions stands as one of the most challenging open problems in incompressible fluid dynamics. We refer the reader for instance to~\cite{DM87} for a more detailed description of the main difficulties arising in the inviscid limit. 

Up to the authors' knowledge, the best available results on the convergence of the vanishing-viscosity scheme assume additional regularity on the limiting solution of the Euler system, see~\cites{K07,Ma07,MR08,R06}.
Our second main result moves in this direction and represents the $L^{\exp}$-counterpart of the inviscid limit proved in~\cite{MR08}*{Theorem~1.3}.

\begin{thm}[Vanishing-viscosity limit] 
\label{t:main2}
For $\nu>0$, let $u^\nu$ be a Leray--Hopf weak solution of~\eqref{navierstokes} and let $U\in X_{u_0}$ be an admissible weak solution of~\eqref{euler}, both with the same initial condition $u_0 \in L^2(\R^d; \R^d)$, $\divergence u_0=0$. 
Let us assume that, for some $\sigma>0$,
\begin{align*}
    \norm{\nabla^s U(t)}_{L^{\exp}(\R^d)} &\leq f(t), \quad \text{for some}\ f \in L^1([0,T]), \\
     \norm{\nabla^s U(t)}_{L^2(\R^d)} + \sup_{\nu \in (0,1)} \sqrt{\nu}\, \norm{\nabla u^\nu (t)}_{L^2(\R^d)} &\leq g(t), \quad \text{for some}\ g \in L^2([0,T]), \\
     \norm{U(t)}_{L^{2+\sigma}(\R^d)} +\sup_{\nu \in (0,1)}\norm{u^\nu(t)}_{L^{2+\sigma}(\R^d)} &\leq h(t), \quad \text{for some}\ h \in L^\infty([0,T]),
\end{align*}
for a.e.\ $t\in[0,T]$.
There exist two constants $\bar\nu>0$ and $M>1$, depending on~$T$, $\sigma$ and the functions~$f$, $g$ and~$h$ above only, with the following property.
If $\nu\in(0,\bar\nu)$, then
\begin{equation}
    \sup_{t\in[0,T]}\norm{u^\nu(t) -U(t)}_{L^2(\R^d)} \leq M\nu^\frac{1}{M} \label{rate 1}.
\end{equation} 
In particular, $u^\nu\to U$ in $L^\infty([0,T];L^2(\R^d;\R^d))$ as $\nu\to0^+$.
\end{thm}

\subsection{Plan of the paper}
The remaining part of our paper is organized as follows. 

In \cref{sec:prel}, we recall the basic theory of Orlicz spaces needed to deal with $L^{\exp}$ functions.
Building on a basic real-variable estimate \cref{convex conj}, we prove a logarithmic interpolation inequality in \cref{orlicz-holder}. This is our key tool and it provides an alternative and more elementary route to the aforementioned~\cite{MR08}*{Theorem~1.1}.

In \cref{sec:proof_weak-strong}, we detail the proof of our first main result \cref{t:main1}, together with the proof of the energy conservation claimed in \cref{energy conservation for U}.

In \cref{sec:proof_inviscid}, we prove our second main result \cref{t:main2} about the convergence of the vanishing-viscosity limit.

\section{Preliminaries}
\label{sec:prel}

In this section, we recall some definitions and preliminary tools which will be used throughout the paper.

\subsection{The space \texorpdfstring{$L^{\exp}$}{Lˆexp} and a \texorpdfstring{$\log$}{log}-interpolation estimate}

We define $\psi(s) = e^s -1$ for all~$s\ge0$ and we denote the corresponding Orlicz space of \emph{exponentially integrable functions} as
\begin{equation}
 L^{\exp}(\R^d) = \left\{ f\colon \R^d\to \R\   :  \exists\beta>0\ \text{such that}\ \int_{\R^d} \psi\left(\frac{\abs{f(x)}}{\beta}\right) \, dx < +\infty\right\}. \nonumber 
\end{equation}
In addition, given $f \in L^{\exp}(\R^d)$, we let
\begin{equation}
\label{eq:def_lux_norm}
    \norm{f}_{L^{\exp}} = \inf\left\{ \beta > 0 \ : \ \int_{\R^d} \psi\left(\frac{\abs{f(x)}}{\beta}\right) \, dx \leq 1 \right\}
\end{equation}
be the \emph{Luxemburg norm} associated to~$\psi$. 
Since~$\psi$ is a \emph{Young function}, by the standard theory of Orlicz spaces (see~\cites{BS88,KK91,RR91} for instance) the set $L^{\exp}(\R^d)$ is a vector space and $(L^{\exp}(\R^d),\norm{\,\cdot\,}_{L^{\exp}})$ is a Banach space.
In addition, the Luxemburg norm satisfies
\begin{equation}
\label{eq:lux_1}
\int_{\R^d} \psi\left( \frac{\abs{f(x)}}{\norm{f}_{L^{\exp}}} \right)\, dx = 1
\end{equation}
for all $f\in L^{\exp}(\R^d)$, $f\ne0$.

It is worth noticing that 
\begin{equation}
\label{exp_in_p}
L^{\exp}(\R^d)\subset L^p(\R^d)
\end{equation}
for all $p\in[1,+\infty)$ with continuous embedding.
Indeed, for any $p\in\mathbb N$, we can easily estimate
\begin{equation*}
    \frac{1}{p!} \int_{\R^d} \left( \frac{\abs{f(x)}}{\norm{f}_{L^{\exp}}} \right)^p \, dx 
\leq 
\int_{\R^d} \psi\left( \frac{\abs{f(x)}}{\norm{f}_{L^{\exp}}} \right)\, dx = 1,
\end{equation*}
whenever $f\in L^{\exp}(\R^d)$, $f\ne0$, so that 
\begin{equation*}
    \norm{f}_{L^p} \leq (p!)^{\frac{1}{p}} \norm{f}_{L^{\exp}}.
\end{equation*}
The John-Nirenberg inequality implies that $\BMO(\R^d)\subset L_{loc}^{\exp}(\R^d)$ and moreover, 
by playing with the function $x\mapsto\log |x|$ (with an appropriate cut-off at infinity), it is also easily seen that $L^{\exp}(\R^d)\not\subset L^\infty(\R^d)$ and  $\BMO(\R^d)\subsetneqq L_{loc}^{\exp}(\R^d)$.

The following result is a particular instance of the well-known \emph{Legendre--Fenchel transformation}, see~\cite{RR91}*{Chapter~1, Theorem~3} for example.
For the reader's convenience, we provide a simple proof of it. 

\begin{lem}[Duality estimate] 
\label{convex conj}
If $s,t\geq 0$, then
\begin{equation}
    st \leq  (e^s -1) +  t \log(t+1). \label{young ineq}
\end{equation}
\end{lem}

\begin{proof}
It is enough to prove that, for any fixed $t_0\geq 0$, the function
\begin{equation*}
g_{t_0}(s)= e^s-1+t_0\log (t_0+1)-st_0,
\quad
s\ge0,
\end{equation*}  
satisfies $g_{t_0}(s)\geq 0$ for all $s\ge0$.
Since $g_{t_0}''(s)=e^s> 0$, the function $g_{t_0}$ is strictly convex and it thus has a unique global minimum in $[0,+\infty)$. 
If $t_0\leq1$, then $g_{t_0}'(s)=e^s-t_0\geq 0$ and so $g_{t_0}$ achieves its minimum at $s=0$, whence $g_{t_0}(0)=t_0\log (1+t_0) \geq 0$.
If $t_0>1$, then $g_{t_0}$ achieves its minimum at some $s_{\min}\in (0,+\infty)$ such that $e^{s_{\min}}=t_0$.
Therefore
\begin{equation*}
g_{t_0}(s_{\min})=t_0-1+t_0\log \left(1+\frac{1}{t_0}\right)\geq 0
\end{equation*}
and the conclusion follows. 
\end{proof}

The simple estimate proved in \cref{convex conj} above implies that the space $L^1(\R^d)\cap L^\infty(\R^d)$ is contained in the dual space of $L^{\exp}(\R^d)$.
More precisely, we have the following result, which reproduces~\cite{MR08}*{Theorem~1.1} via a much simpler and more elementary approach.

\begin{prop}[$\log$-interpolation estimate] 
\label{orlicz-holder}
There exists a constant $C>0$ with the following property.
If $f \in L^{\exp}(\R^d)$ and $g \in L^1(\R^d) \cap L^\infty(\R^d)$, then
\begin{equation}
   \int_{\R^d} \abs{f(x)\,g(x)}\, dx \leq C \norm{f}_{L^{\exp}} \norm{g}_{L^1} \Big[ \log\left(1+ \norm{g}_{L^\infty} \right) + \big|\log\, \norm{g}_{L^1}\big|  +1 \Big]. \label{orlicz-holder ineq}
\end{equation}
\end{prop}

\begin{proof}
Without loss of generality, we can assume that both $f$ and $g$ do not vanish identically. 
Letting $\lambda>0$, by~\eqref{young ineq}, in combination with~\eqref{eq:def_lux_norm} and~\eqref{eq:lux_1}, we can estimate
\begin{align*}  
\int_{\R^d} & \abs{f(x)\,g(x)}\ dx  
= 
\lambda \,\norm{f}_{L^{\exp}} \int_{\R^d} \abs*{\frac{f(x)}{\norm{f}_{L^{\exp}}}\, \frac{g(x)}{\lambda} }\, dx 
     \\ 
& \leq  
\lambda\, \norm{f}_{L^{\exp}} \left[ \int_{\R^d} \left(e^{\abs{f(x)} / \norm{f}_{L^{\exp}}}- 1\right)\, dx +  \int_{\R^d} \frac{\abs{g(x)}}{\lambda} \,\log\left(\frac{\abs{g(x)}}{\lambda} + 1\right) \, dx   \right]
     \\ 
& =  
\lambda\,\norm{f}_{L^{\exp}}  \left[ 1 +  \int_{\R^d} \frac{\abs{g(x)}}{\lambda} \log\left(\frac{\abs{g(x)}}{\lambda} + 1\right) \, dx \right]. 
\end{align*}
Since clearly 
\begin{align*}
\int_{\R^d} \frac{\abs{g(x)}}{\lambda} 
&
\log\left(\frac{\abs{g(x)}}{\lambda} + 1\right) \, dx
=
\left[
\int_{\abs{g(x)} \leq \lambda} 
+
\int_{\abs{g(x)} > \lambda}
\right] 
\,
\frac{\abs{g(x)}}{\lambda} \log\left(\frac{\abs{g(x)}}{\lambda} + 1\right) \, dx 
     \\ 
& \leq  
\frac{\norm{g}_{L^1}}{\lambda} \log 2  + \int_{\abs{g(x)} > \lambda} \frac{\abs{g(x)}}{\lambda} \log\left(\frac{\abs{g(x)}}{\lambda} + 1\right) \,dx
\end{align*}
and
\begin{align*}
    \log\left( \frac{\abs{g(x)}}{\lambda} + 1 \right) & = \log (\abs{g(x)} +\lambda) - \log \lambda \leq \log (1+2 \norm{g}_{L^\infty} ) + \abs{\log\lambda}
\end{align*}
whenever $\abs{g(x)} >\lambda$, we get that 
\begin{equation*}
    \int_{\R^d} \abs{f(x)\, g(x)}\ dx \leq  
\lambda \norm{f}_{L^{\exp}} 
\left[ 
1 + \frac{\norm{g}_{L^1}}{\lambda} 
\Big( 
\log 2 + \log(1+ 2 \norm{g}_{L^\infty} ) + \abs{\log\lambda} 
\Big) 
\right]. 
\end{equation*}
Choosing $\lambda = \norm{g}_{L^1}$, the conclusion readily follows. 
\end{proof}

\section{Proof of \texorpdfstring{\cref{t:main1}}{weak-strong uniqueness}}
\label{sec:proof_weak-strong}

In this section, we prove \cref{t:main1}.
To this aim, we need some preliminary results. We begin with the following simple differential identity involving the symmetric gradient of a function.
Here and in the following, we let $p'$ be the conjugate exponent of $p\in[1,+\infty]$, so that $\frac1p+\frac1{p'}=1$. 
In addition, we adopt Einstein's convention on summing repeated indices. 

\begin{lem}[Differential identity] 
\label{third term}
Let $\sigma > 0$. 
If $U \in L^{2+\sigma}(\R^d;\R^d)$ is such that 
\begin{equation*}
\nabla^s U \in L^{(1+\frac{\sigma}{2})'}(\R^d;\R^{d\times d}),
\qquad
\divergence{U} = 0,
\end{equation*}
then 
\begin{equation}
    \partial_i (U^i U^j) + \partial_j \left(\frac{U^i U^i}{2}\right) = 2U^i (\nabla^s U)^{ij}
\label{distr symm gradient}
\end{equation}
for every $j=1,\dots,d$ in the sense of distributions.
\end{lem}

\begin{proof}
Let $\delta>0$ and let $U_\delta$ be the mollification of $U$. 
Since $U_\delta$ is smooth and $\divergence U_\delta=0$, we can compute
\begin{equation}
\label{delta_identity}
    \partial_i (U_\delta^i U_\delta^j) + \partial_j \left( \frac{U_\delta^i U_\delta^i}{2}\right) = U_\delta^i \, \partial_i U_\delta^j + U_\delta^i \, \partial_j U_\delta^i = 2U_\delta^i (\nabla^s U_\delta)^{ij} 
\end{equation}
proving the (pointwise) validity of~\eqref{distr symm gradient} for~$U_\delta$.
Now, by multiplying~\eqref{delta_identity} above by a test function $\phi \in C^\infty_c(\R^d)$ and integrating by parts in the left hand side, we get
\begin{equation*}
    -\int_{\R^d} 
\left( 
U_\delta^i U_\delta^j \,\partial_i \phi +  \frac{U_\delta^i U_\delta^i}{2} \,\partial_j \phi 
\right) \, dx = 2\int_{\R^d} U_\delta^i (\nabla^s U)^{ij}_\delta \varphi \, dx. 
\end{equation*}
Letting $\delta\to0^+$ and observing that 
\begin{equation*}
U_\delta \to U\ \text{in}\ L^{2+\sigma}(\R^d;\R^d),
\qquad
\nabla^s U_{\delta}\to \nabla^s U\ \text{in}\ L^{(1+\frac{\sigma}{2})'}(\R^d;\R^{d\times d}),
\end{equation*}
we easily conclude that 
\begin{equation}
    -\int_{\R^d} 
\left( U^i U^j \,\partial_i \phi +  \frac{U^i U^i}{2}\, \partial_j \phi 
\right) \,dx = 2\int_{\R^d} U^i (\nabla^s U)^{ij} \,\phi \, dx \label{distr symm grad 1}
\end{equation}
and the proof is complete.
\end{proof}

As we already pointed out, the key step towards the proof of \cref{t:main1} is to establish inequality~\eqref{dis_bella}. 
However, it is useful to
prove an energy conservation result for the solution $U$ which is apparently transversal to the usual \emph{Onsager's condition} $L^3_t(B^{1/3}_{3,c_0})$ from    \cite{CCFS08}.

\begin{rem}
The fact that $U$ conserves the energy plays a crucial role to conclude the proof 
of \cref{t:main1}. Indeed, after deducing that $E_{\rel}\equiv 0$ in some small time interval $(0,\delta)$, one needs to iterate the procedure in the interval $(\delta,2\delta)$ and more generally in every $(i\delta,(i+1)\delta)$ for $i=1,\dots,N$, up to covering the whole interval $(0,T)$. To fix the ideas, we focus on the first step in passing from $(0,\delta)$ to $(\delta,2\delta)$. 
In order to make the iteration possible, one has to check that $u$ is also an admissible weak solution in the interval $(\delta,2\delta)$ which is not true for a general admissible weak solution in $(0,2\delta)$. In our case, this holds since in the previous step we have shown that $u\equiv U$ in $(0,\delta)$. Thus, if $U$ conserves the energy, then $u$ does too. In particular, the admissibility of $u$ in the interval $(0,T)$ transfers to the admissibility in $(\delta,2\delta)$. 
\end{rem}

\begin{prop}[Energy conservation]
\label{energyconsforU}
Let $\sigma>0$ and let
\begin{equation*}
U\in L^\infty([0,T]; L^2\cap L^{2+\sigma} (\R^d; \R^d)))    
\end{equation*}
be a weak solution of~\eqref{euler}.
If 
\begin{equation*}
 \nabla^s U\in L^1\big([0,T]; L^{(1+\frac{\sigma}{2})'}(\R^d; \R^{d \times d}) \big),   
\end{equation*}
then
\begin{equation}\label{En_bal_U}
\int_{\R^d} |U(x,\tau)|^2\,dx
=
\int_{\R^d} |U(x,0)|^2\,dx
\end{equation}
for all $\tau >0$. 
\end{prop}

\begin{proof}
Let $\e>0$ and let $U_\e$ be the space mollification of $U$. 
We clearly have that
\begin{equation*}
U_\e \in \Lip([0,T]\times\R^d;\R^d)\cap\Lip([0,T]; L^2(\R^d;\R^d) \cap W^{1,\infty}(\R^d;\R^d)).
\end{equation*}
Indeed, it is easy to check that
\begin{equation*}
x\mapsto U_\e(x,t) \in C^1(\R^d;\R^d) \cap L^2(\R^d;\R^d) \cap W^{1, \infty}(\R^d;\R^d)
\end{equation*}
for every $t \in (0, T)$, so that the Lipschitz regularity with respect to the time variable can be recovered from the equation
\begin{equation*}
    \partial_t U_\e =- \divergence{\left( U\otimes U\right)}_\e - \nabla P_\e,
\end{equation*}
where $P_\e$ is the corresponding regularized scalar pressure. 
Now, by mollifying the momentum equation in the first line of the system~\eqref{euler} and by multiplying by~$U_\e$, we readily obtain the well-known local energy balance
$$
\partial_t \left(\tfrac{|U_{\e} |^2}{2}\right) +\divergence{\left(\left(\tfrac{|U_{\e} |^2}{2}+P_\e \right)U_\e \right)}=\divergence{\left(R^U_\e U_\e\right)}-R^U_\e : \nabla U_\e,
$$
where
\begin{equation*}
R^U_\e=U_{\e}\otimes U_{\e}-\left( U\otimes U\right)_\e   
\end{equation*}
stands for the usual commutator. 
Note that the above computations make sense in virtue of the regularity of ~$U_\e$. 
By integrating in space and time, we obtain that~$U_\e$ obeys the energy balance 
\begin{equation}
\label{energy_balance_Ueps_prel}
E_{U_{\e}}(\tau)-E_{U_{\e}}(0)=-\int_0^\tau\int_{\R^d}R^U_{\e} :\nabla U_{\e} \,dx \, dt=-\int_0^\tau \int_{\R^d}R^U_{\e} :\nabla^s U_{\e} \,dx \, dt, 
\end{equation}
where in the last equality we exploited the fact that $R^U_\e$ is a symmetric matrix. 
As a consequence, we get that 
\begin{equation}\label{energy_balance_Ueps}
\left| E_{U_{\e}}(\tau)-E_{U_{\e}}(0)\right|\leq \int_0^\tau \int_{\R^d}|R^U_{\e}|\, \left|\nabla^s U_{\e}\right| \,dx\,dt,
\end{equation}
where we set $E_U(\tau)=\frac12\int_{\R^d}|U(\tau)|^2\,dx$ for all $\tau>0$.
Now, since $U(\cdot, \tau)\in L^2_x$ for every $\tau>0$, the left-hand side of~\eqref{energy_balance_Ueps} converges to $| E_{U}(\tau)-E_{U}(0)|$ as $\e\to0^+$. 
Thus, we just have to prove that the right-hand side of~\eqref{energy_balance_Ueps} vanishes as $\e \to 0^+$.
To this aim, notice that we can estimate 
\begin{align*}
	\int_{\R^d}|R^U_{\e}(t)|\, \left|\nabla^s U_{\e}(t)\right|\,dx & \leq \norm{R^U_{\e}(t)}_{L_x^{1+\frac{\sigma}{2}}}\| \nabla^s U_{\e}(t)\|_{L_x^{(1+\frac{\sigma}{2})'}}
	\\ & \leq 2 \|U(t)\|^2_{L_x^{2+\sigma}}\| \nabla^s U(t)\|_{L_x^{(1+\frac{\sigma}{2})'}}\in L^1([0, \tau])
\end{align*}
for all $t \in (0, \tau)$.
In addition, recalling that
\begin{equation*}
U(\cdot,t)\in L^{2+\sigma}_x,
\quad
\left(U\otimes U\right)(\cdot, t)\in L_x^{1+\frac{\sigma}{2}},
\end{equation*} 
for $t\in[0,T]$,
we can further estimate
$$
\norm{R^U_{\e}(t)}_{L_x^{1+\frac{\sigma}{2}}}\leq \norm{\left(U\otimes U\right)_{\e}(t)-\left(U\otimes U\right)(t)}_{L_x^{1+\frac{\sigma}{2}}}+2\|U\|_{L^\infty_{t} \left(L_x^{2+\sigma}\right)}
\,
\|U(t)-U_{\e}(t)\|_{L_x^{2+\sigma}}
$$
for $t\in[0,T]$.
Since the right-hand side of the above inequality goes to $0$ as $\e \to 0^+$ for a.e.\ $t \in (0, \tau)$, by the Dominated Convergence Theorem we conclude that the right hand side in \eqref{energy_balance_Ueps} vanishes as $\e\to0$. 
The proof is complete.
\end{proof}

\begin{rem}[Solutions of~\eqref{euler} as in \cref{energyconsforU} are in $C({[}0,T{]};L^2)$] 
\label{strong continuity in time}
Under the assumptions of \cref{energyconsforU}, we infer that $U \in C([0, T]; L^2(\R^d; \R^d))$. Indeed, the map $t\to U(t)$ is weakly continuous in $L^2$ and, in virtue of \cref{energy conservation for U}, $t\mapsto U(t)$ has constant $L^2$-norm. 
\end{rem}

We are in the position to prove the validity of inequality~\eqref{dis_bella}.

\begin{prop} [Key estimate on the relative energy] \label{relative energy estimate}
Under the assumptions of \cref{t:main1}, inequality~\eqref{dis_bella} holds for every $\tau \in [0, T]$.
\end{prop}

\begin{proof}
As already discussed above, we would like to use (a regularization of) the solution~$U$ as a test function in the weak formulation of~$u$.
To this aim, let $\e>0$ and let $U_\e$ be the  space mollification of $U$. 
As already remarked in the proof of \cref{energyconsforU}, we know that
\begin{equation*}
U_\e \in \Lip([0,T]\times\R^d;\R^d)\cap\Lip([0,T]; L^2(\R^d;\R^d) \cap W^{1,\infty}(\R^d;\R^d)).
\end{equation*}
For any $\e>0$, we let 
\begin{equation}\nonumber
    E_{\rel}^{\e} (\tau) = \frac{1}{2} \int_{\R^d} \abs{u(x, \tau) - U_{\e}(x, \tau)}^2 \, dx
    \label{E-e}
\end{equation}
and we notice that 
\begin{equation}
    E_{\rel}^\e (\tau) = \frac{1}{2} \int_{\R^d} \abs{u(x, \tau)}^2 \, dx + \frac{1}{2} \int_{\R^d} \abs{U_\e(x, \tau)}^2 \, dx - \int_{\R^d} u(\tau) \cdot U_\e(\tau)\, dx. \label{E-e 1}
\end{equation}
for all $\tau\in[0,T]$.
We now introduce the usual commutator 
$$R_\e^U = U_\e \otimes U_\e - (U \otimes U)_\e. $$
As observed in the proof of \cref{energyconsforU}, $U_\e$ satisfies the energy balance \eqref{energy_balance_Ueps_prel}, where now $U_{\e}(0) = (u_0)_\e$. Since $u$ is an admissible weak solution according to \cref{admissible solution}, by combining~\eqref{E-e 1} and~\eqref{energy_balance_Ueps_prel}, we can estimate
\begin{align}
    E_{\rel}^\e (\tau) & \leq \frac{1}{2} \int_{\R^d} \abs{u_0}^2 \, dx + \frac{1}{2}\int_{\R^d} \abs{(u_0)_\e}^2 \, dx  - \int_0^\tau \int_{\R^d} R_\e^U \colon \nabla^s U_\e \, dx \, dt - \int_{\R^d} u(\tau) \cdot U_\e(\tau) \, dx \nonumber
    \\ & = \frac{1}{2} \int_{\R^d} \abs{u_0 - (u_0)_\e}^2 \, dx - \int_0^\tau \int_{\R^d} R_\e^U \colon \nabla^s U_\e \, dx\, dt + \int_{\R^d} u_0 \cdot (u_0)_\e  \,dx \nonumber
    \\ & \hspace{0.4 cm} - \int_{\R^d} u(\tau) \cdot U_\e(\tau)\,  dx   \nonumber
    \\ & = \frac{1}{2} \int_{\R^d} \abs{u_0 - (u_0)_\e}^2 \, dx - \int_{0}^\tau \int_{\R^d} R^U_\e \colon \nabla^s U_\e  \,dx \, dt \nonumber
    \\ & \hspace{0.4 cm } - \int_0^\tau \int_{\R^d} ( \partial_t U_\e \cdot u + u \otimes u \colon \nabla U_\e) \, dx\,  dt  \nonumber.
\end{align}
By standard manipulations, we can rewrite the last term of the above chain as 
\begin{align*}
- \int_0^\tau \int_{\R^d} ( \partial_t U_\e \cdot u &+ u \otimes u \colon \nabla U_\e) \, dx\,  dt 
\\
& = 
 \int_0^\tau \int_{\R^d} ( \divergence{(U_\e \otimes U_\e)} \cdot u -\divergence{ R_\e^U } \cdot u + \nabla P_\e\cdot u - u\otimes u \colon \nabla U_\e) \, dx\,  dt  
    \\ 
& =  -\int_0^\tau \int_{\R^d}  u\cdot \divergence{R_\e^U}  \, dx\, dt 
+ \int_0^\tau \int_{\R^d} \nabla U_\e \colon (U_\e \otimes u - u \otimes u) \,dx\,  dt 
    \\ 
& = 
- \int_0^\tau \int_{\R^d}u\cdot \divergence{R_\e^U} \,dx \, dt 
+ \int_0^\tau \int_{\R^d} \nabla^s U_\e \colon (U_\e-u) \otimes (u-U_\e)\, dx \, dt
\end{align*}
where, as usual, $P_\e$ stands for the corresponding convoluted scalar pressure. 
By combining the above estimates, we conclude that 
\begin{equation}\label{est_E_rel_eps_fin}
\begin{split}
E_{\rel}^\e  (\tau) 
&\le
\frac{1}{2} \int_{\R^d} \abs{u_0 - (u_0)_\e}^2  \,dx 
- \int_{0}^\tau \int_{\R^d} R^U_\e \colon \nabla^s U_\e  \,dx  \,dt
\\
&\quad - \int_0^\tau \int_{\R^d}u\cdot \divergence{R_\e^U} \,dx \, dt 
+ \int_0^\tau \int_{\R^d} \nabla^s U_\e \colon (U_\e-u) \otimes (u-U_\e)\, dx \, dt.
\end{split}
\end{equation}
Now we need to justify how to pass to the limit as~$\e\to0^+$. 
Since $u_0, U(\cdot, \tau)\in L^2_x$, we clearly have that 
\begin{equation*}
    \lim_{\e\to 0^+} E_{\rel}^\e(\tau) = E_{\rel}(\tau), 
\qquad
\lim_{\e\to 0^+} \int_{\R^d} \abs{u_0 - (u_0)_\e}^2 \ dx = 0.
\end{equation*}
Since $\nabla^s U(\cdot,t) \in L_x^{\exp}$, thanks to~\eqref{exp_in_p} we surely have that $\nabla^s U(\cdot,t) \in L_x^{(1+\frac{\sigma}{2})'}$, so that 
\begin{equation*}
   \norm{\nabla^s U_\e(t)}_{L_x^{(1+\frac{\sigma}{2})'}} \leq \norm{\nabla^s U(t)}_{L_x^{(1+\frac{\sigma}{2})'}} \leq C(\sigma)\, \norm{\nabla^s U(t)}_{L_x^{\exp}} \in L^1([0,T]).
\end{equation*}
Thus, as already shown in the proof of \cref{energyconsforU}, we know that 
\begin{equation*}
    \lim_{\e \to 0^+}\int_0^\tau \int_{\R^d} R_\e^U \colon \nabla^s U_\e \,dx\,dt=0.
\end{equation*}
At this point, to conclude the proof, we are only left to show that
\begin{equation}
\lim_{\e \to 0^+}\int_0^\tau \int_{\R^d} \nabla^s U_\e \colon (U_\e-u) \otimes (u-U_\e )\,dx\,dt
=
\int_0^\tau \int_{\R^d} \nabla^s U\colon (U-u) \otimes (u-U)\,dx\,dt
\label{termine3}
\end{equation}
and
\begin{equation}
\lim_{\e \to 0^+}\int_0^\tau \int_{\R^d}u\cdot \divergence{R_\e^U} \,dx\,dt =0. 
\label{termine2}
\end{equation}
We now deal with each limit separately, the most delicate one being~\eqref{termine2}.

\smallskip

\textit{Proof of~\eqref{termine3}}. 
As remarked above, we have that $\nabla^s U \in L^1_t\left( L_x^{(1+\frac{\sigma}{2})'}\right)$, so that $\nabla^s U_\e(\cdot,t)\to \nabla^s U(\cdot,t)$ in $L_x^{(1+\frac{\sigma}{2})'}$ as $\e\to 0^+$. Moreover, since $U(\cdot,t) \in L_x^{2+\sigma}$, we know that  $U_\e(\cdot,t) \to U(\cdot,t)$ in $L_x^{2+\sigma}$. Hence, we have that 
\begin{equation*}
(U_\e-u) \otimes (u-U_\e  )(\cdot,t) \to (U-u)\otimes(u-U)(\cdot,t)
\quad
\text{in}\
L_x^{1+\frac{\sigma}{2}}.
\end{equation*}
Combining the above limits, we infer that 
$$\nabla^s U_\e: (U_\e-u) \otimes (u-U_\e )( \cdot\,, t) \to \nabla^s U: (U-u) \otimes (u-U)( \cdot\,, t) \quad \text{in}\ L^1_x. $$ 
In addition, by Holder's inequality, for every $\e> 0$ we have that 
\begin{align*}
 \norm{\nabla^s U_\e: (U_\e-u) &\otimes (u-U_\e )(t)}_{L_x^1} 
 \leq  \norm{\nabla^s U_\e(t)}_{L_x^{(1+\frac{\sigma}{2})'}} \norm{\left(u-U_\e\right)(t)}_{L_x^{2+\sigma}}^2 
\\ & \leq C(\sigma) \norm{\nabla^s U(t)}_{L_x^{\exp}} \norm{u}_{L_t^\infty(L_x^{2+\sigma})}^2 \norm{U}_{L_t^\infty(L_x^{2+\sigma})}^2
\in L^1([0,T]).
\end{align*} 
By combining all the bounds given above, we get~\eqref{termine3} by the Dominated Convergence Theorem.

\smallskip

\textit{Proof of~\eqref{termine2}}.
Fix $t \in (0, \tau)$.
Since $x\mapsto U_\e(x,t)$ is smooth enough and $\divergence u=0$, we can write
\begin{equation}
\begin{split}
    \int_{\R^d} \divergence{R_\e^U} \cdot u \, dx 
& = \int_{\R^d} \left( U_\e^i\, \partial_i U_\e^j - \partial_i(U^i U^j)_\e \right) u^j \, dx 
    \\ 
& = 
\int_{\R^d} \left( U_\e^i\, (\partial_i U_\e^j + \partial_j U_\e^i) - \partial_i (U^i U^j)_\e - \partial_j \left( \frac{U^i U^i}{2}\right)_{\e}\, \right) u^j \, dx 
    \\ 
& = 
\int_{\R^d} \left( 2U_\e^i\, ( \nabla^s U_{\e})^{ij} - \partial_i (U^i U^j)_\e - \partial_j \left( \frac{U^i U^i}{2}\right)_\e \,\right) u^j \, dx
\\
    &=
2\int_{\R^d}\left( U_\e^i\, ( \nabla^s U_{\e})^{ij} -\left( U^i \,(\nabla^s U)^{ij} \right)_\e \,\right) u^j \, dx,
\end{split} 
\label{last formula to check}
\end{equation}
where in the last equality we exploited \cref{third term}.
In addition, since $U_\e(\cdot,t) \to U(\cdot,t)$ in $L_x^{2+\sigma}$ and $\nabla^s U_\e(\cdot,t)\to \nabla^s U(\cdot,t)$ in $L_x^{(1+\frac{\sigma}{2})'}$, we get that 
\begin{equation*}
(U_\e\cdot \nabla^s U_\e)(\cdot,t) - (U \cdot \nabla^s U)_\e(\cdot,t) \to 0
\
\text{in}\ 
L_x^{(2+\sigma)'}.
\end{equation*}
All in all, since $u \in L^{2+\sigma}_x$, in virtue of~\eqref{last formula to check} we infer that 
\begin{equation*}
    \lim_{\e \to 0^+} \int_{\R^d} u\cdot \divergence{R^U_\e}  \ dx = 0
\end{equation*}
for a.e.\  $t\in (0,\tau)$.
To conclude, we notice that, for $\e>0$ and $t \in (0, \tau)$, we can use~\eqref{last formula to check} to estimate
\begin{align*}
    \abs*{\int_{\R^d}u(t) \cdot \divergence{R^U_\e(t)}  \, dx } &\leq \norm{u(t)}_{L_x^{2+\sigma}}\|U(t)\|_{L_x^{2+\sigma}} \norm{\nabla^s U(t)}_{L_x^{(1+\frac{\sigma}{2})'}} \\
    &\leq C(\sigma)\, \norm{u}_{L^\infty_t(L_x^{2+\sigma})}\|U\|_{L^\infty_t(L_x^{2+\sigma})}  \norm{\nabla^s U(t)}_{L_x^{\exp}}
\in L^1([0,T]).
\end{align*}
We thus get~\eqref{termine3} by the Dominated Convergence Theorem and the proof is complete.
\end{proof}

\begin{rem}[Uniform bound on relative energy for small times] 
\label{uniform bound}
As a by-product of \cref{relative energy estimate}, we can obtain a uniform bound on the relative energy for small times, which will be crucial in the proof of \cref{t:main1}. 
Precisely, under the assumptions of \cref{t:main1}, from inequality~\eqref{dis_bella} we get that
\begin{align*}
    E_{\rel} (\tau) & \leq \int_0^\tau \norm{\nabla^s U(t)}_{L^{(1+\sigma/2)'}_x} \norm{U(t) - u(t)}_{L_x^{2+\sigma}}^2 \, dt
\end{align*}
for any $\tau>0$.
Consequently, for any given $ \eta>0 $ we can find a time $t_\eta\in(0,T]$, depending on $\eta$ only, such that 
\begin{equation}
    \sup_{\tau\in[0,t_\eta]} E_{\rel}(\tau) \leq \eta. \label{uniform bound estimate} 
\end{equation}
\end{rem}

We are now ready to prove \cref{t:main1} by combining \cref{relative energy estimate} together with some basic integral inequality arguments. 
Our strategy is very similar to that of~\cite{MR08}, although some adaptions are needed in order to deal with our weaker integrability assumption $\nabla^sU\in L^{\exp}_x$ and the weak-strong uniqueness result.

\begin{proof} [Proof of \cref{t:main1}]
We divide the proof in three steps.

\smallskip 

\textit{Step~1: integral inequality for the relative energy}.
By \cref{relative energy estimate}, for every $\tau \in (0, T)$ we have that 
\begin{equation}
    E_{\rel}(\tau) \leq  \int_0^\tau \int_{\R^d} \nabla^s U \colon (U-u) \otimes (u-U) \, dx \, dt. \label{starting point}
\end{equation}
Let us set 
\begin{equation*}
\alpha(x,t) = \abs*{U(x,t)-u(x,t)}^2,
\qquad
(t,x)\in[0,T]\times\R^d
\end{equation*}
and, for any given $m>0$,
\begin{equation}
\label{alpha_rotto}
\alpha_m^+ = \alpha\, \mathbf{1}_{\set*{\alpha > m}},
\qquad
\alpha_m^- = \alpha\, \mathbf{1}_{\set*{\alpha\le m}}.
\end{equation}
By \cref{orlicz-holder}, we can exploit~\eqref{starting point} to estimate
\begin{equation}
\label{cia1}
\begin{split}
    \|\alpha(\tau)\|_{L_x^1} 
 &\leq 
2 \int_0^\tau \int_{\R^d} \alpha\, \abs{\nabla^s U} \, dx\, dt 
    \\ 
& = 
2 \int_0^\tau \int_{\R^d} \alpha_m^-\,\abs{\nabla^s U} \, dx \, dt 
+ 
2 \int_0^\tau \int_{\R^d} \alpha_m^+\, \abs{\nabla^s U}\, dx \, dt 
    \\ 
& \leq 
C \int_0^\tau \norm{\nabla^s U(t) }_{L_x^{\exp}}\, \norm{\alpha_m^-(t)}_{L_x^1} \bigg(\, \abs*{ \log\norm{\alpha_m^-(t)}_{L_x^1}} + \log\big(\norm{\alpha_m^-(t)}_{L_x^\infty} + 1\big) +1 \bigg)\, dt 
\\
& \quad + 2
\int_0^\tau \int_{\R^d} \alpha_m^+\, \abs{\nabla^s U}\, dx \, dt
\end{split}
\end{equation}
for all $\tau\in(0,T)$. Recalling the definition in~\eqref{alpha_rotto}, we have that 
\begin{equation*}
\norm{\alpha^-_m(t)}_{L_x^\infty}\le m
\end{equation*} 
for all $t\in[0,T]$. Therefore, observing that the function $r\mapsto r\,(|\log r|+\log (m+1) +1)$ is non-decreasing for all $r\ge0$ whenever $m>0$, 
we can estimate 
\begin{equation}
\label{cia2}
\begin{split}
\norm{\alpha_m^-(t)}_{L_x^1} 
&
\bigg(\abs*{ \log\norm{\alpha_m^-(t)}_{L_x^1}} 
+
\log\big(\norm{\alpha_m^-(t)}_{L_x^\infty} + 1\big) +1 \bigg)
\\
&\le
\norm{\alpha_m^-(t)}_{L_x^1} \bigg( \abs*{ \log\norm{\alpha_m^-(t)}_{L_x^1}} 
+
\log(m + 1) +1 \bigg)
\\
&\le 
\norm{\alpha(t)}_{L_x^1} \bigg(\abs*{ \log\norm{\alpha(t)}_{L_x^1}} + \log(m + 1) +1 \bigg),
\end{split}
\end{equation}
for all $t\in[0,T]$. On the other side, since $\nabla^s U (\cdot,t) \in L_x^{\exp}$, by~\eqref{exp_in_p} we surely have $\nabla^s U(\cdot,t) \in L_x^{(1+\sigma/4)'}$.
Moreover, since $u,U\in X_{u_0}$, we have 
\begin{equation*}
\alpha \in L^\infty([0,T];L^1(\R^d)) \cap L^\infty([0,T];L^{1+\frac{\sigma}{2}}(\R^d)),
\end{equation*}
so that
\begin{equation}
\label{fianco}
\alpha \in L^\infty ([0,T];L^{1+ \frac{\sigma}{4}}(\R^d))
\end{equation}
by interpolation.
Therefore, recalling the definition of~$\alpha_m^+$ in~\eqref{alpha_rotto}, we can estimate
\begin{equation}
\label{cia3}
\begin{split}
   \int_{\R^d} \alpha_m^+(t,x)&\, \abs{\nabla^s U(t,x)} \, dx 
\leq 
\norm{\alpha_m^+(t)}_{L_x^{1+\frac{\sigma}{4}}}\, \norm{\nabla^s U(t)}_{L_x^{(1+\frac{\sigma}{4})'}} 
    \\ 
& \leq 
C(\sigma) \, \abs{\{\alpha(t,\,\cdot\,)> m\} }^{\frac{2\sigma}{(2+\sigma)(4+\sigma)}}\, \norm{\alpha(t)}_{L_x^{1+\frac{\sigma}{2}}} 
\,\norm{\nabla^s U(t)}_{L_x^{\exp}}
    \\ 
& \leq 
C(\sigma)\,m^{-\frac{\sigma}{4+\sigma}} \, \norm{\alpha(t)}_{L_x^{1+\frac{\sigma}{2}}}^{1+\frac{\sigma}{4+\sigma}} 
\, \norm{\nabla^s U(t)}_{L_x^{\exp}}
\end{split}
\end{equation}
for all $t\in(0,T)$. 
All in all, by putting~\eqref{cia2} and~\eqref{cia3} into~\eqref{cia1}, we get that  
\begin{align}
    \norm{\alpha(\tau)}_{L_x^1} & \leq C \int_0^\tau \norm{\nabla^s U(t)}_{L^{\exp}_x} \norm{\alpha(t)}_{L_x^1} \left( \abs{ \log(\norm{\alpha(t)}_{L_x^1})} + \log(m + 1) +1 \right)  dt 
    \\ & \quad + C(\sigma)\, m^{-\frac{\sigma}{4+\sigma}}\int_0^\tau \norm{\nabla^s U(t)}_{L^{\exp}_x}  \norm{\alpha(t)}_{L_x^{1+\frac{\sigma}{2}}}^{1+\frac{\sigma}{4+\sigma}}  dt \nonumber
    \\ & \leq m^{-\frac{\sigma}{4+\sigma}}\int_0^T f(t) \,dt  +  \int_0^\tau f(t)\,  \norm{\alpha(t)}_{L_x^1} \left( \ \left| \log(\norm{\alpha(t)}_{L_x^1})\right| + \log(m + 1) +1 \right)  dt
    \label{def funzione f}
\end{align}
for all $\tau\in[0,T]$, where we denoted
$$f(t) = \norm{\nabla^s U(t)}_{L_x^{\exp}} \left(C+C(\sigma) \norm{\alpha(t)}_{L_x^{1+\frac{\sigma}{2}}}^{1+\frac{\sigma}{\sigma+4}} \right) \in L^1([0, T]). $$

\smallskip 

\textit{Step~2:  Gr\"onwall's inequality}.
Letting 
\begin{equation*}
y(\tau)= \norm{\alpha(\tau)}_{L_x^1},
\quad
C_T=\int_0^T f(t)\,dt,
\quad
\theta=\frac{\sigma}{4+\sigma},
\quad
C_m=(1+m)e,
\end{equation*}
we can equivalently rewrite inequality~\eqref{def funzione f} as 
\begin{equation}\label{dis_per_y}
y(\tau)\leq \frac{C_T}{m^{\theta}}+\int_0^\tau f(t)y(t)\big( \left| \log y(t)\right| + \log C_m \big)\,dt
\end{equation}
for all $\tau\in[0,T]$.
Now define $\tilde y_m(\tau)=y(\tau)+\frac{1}{m}$. 
In view of \cref{uniform bound} we can pick a time $t_0>0$ such that $y(\tau)\leq \frac12$ for every $\tau\in [0,t_0]$. 
In particular, for $m>2$ we get 
\begin{equation}\label{bound_per_ytilde}
    \frac{1}{m}\leq \tilde y_m(\tau)\leq 1 \quad \text{for all}\  \tau\in [0,t_0].
\end{equation}
Possibly choosing $t_0>0$ even smaller, we can also assume
that 
\begin{equation}
\label{controllo_int_f}
\int_0^{t_0}f(t)\,dt<\frac\theta4.    
\end{equation}
We remark that the choice of $t_0$ realizing~\eqref{bound_per_ytilde} and~\eqref{controllo_int_f} does not depend on $m$.
Since the function $y\mapsto y\left( \left| \log y\right| +\log C_m\right)$ is increasing for each $m>0$, possibly enlarging the constant $C_T$ from line to line in what follows, from~\eqref{dis_per_y} and~\eqref{bound_per_ytilde} we obtain
\begin{align*}
    \tilde y_m(\tau)&\leq \frac{1}{m}+\frac{C_T}{m^\theta}+\int_0^\tau f(t) \,\tilde y_m(t)\left( \left| \log \tilde y_m(t)\right|+\log C_m\right)\,dt\\
    &\leq \frac{C_T}{m^\theta}+\int_0^\tau f(t) \,\tilde y_m(t)\log  (m \,C_m)\,dt
\end{align*}
for all $\tau\in[0,t_0]$.
By Gr\"onwall's inequality, we thus deduce that
$$
\tilde y_m(\tau)\leq \frac{C_T}{m^\theta} e^{ \log (m \, C_m) \int_0^\tau f(t)\,dt }\leq C_T\,m^{-\theta+2\int_0^\tau f(t)\,dt}
\le
C_Tm^{-\frac\theta2}
$$
for all $\tau\in[0,t_0]$, in virtue of~\eqref{controllo_int_f}.
Therefore, recalling that $y=\tilde y_m-\frac{1}{m}$ and letting $m\to+\infty$, we conclude that 
\begin{equation}\label{ynulla}
    y(\tau)= 0 \quad \text{for all}\ \tau\in [0,t_0].
\end{equation}

\smallskip 

\textit{Step~3: iteration argument}.
Let $t_0 \in (0, T) $ be the time fixed at the end of the previous step. 
Note that $t_0$ uniquely depends on~$\sigma$ and on suitable integral norms of~$u$, $U$ and~$\nabla^s U$.  
In order to iterate the argument of Step~2 in the time interval $(t_0,2t_0)$, we need to check that $u$ is still an admissible weak solution in this interval.
In virtue of~\eqref{ynulla}, we know that $u\equiv U$ in $(0,t_0)$.
Therefore, since~$U$ is energy conservative by \cref{energyconsforU} and since $u$ is admissible in the whole time interval $(0,T)$, we can infer that 
$$
\int_{\R^d}|u(t,x)|^2\,dt
\leq 
\int_{\R^d}|u_0(x)|^2\,dt
=
\int_{\R^d}|U(0,x)|^2\,dt
=
\int_{\R^d}|U(t_0,x)|^2\,dt
=
\int_{\R^d}|u(t_0,x)|^2\,dt
$$
for all $t\in[0,T]$.
Thus $u$ is admissible in $(t_0,2t_0)$, so the argument of Step~2 used to prove~\eqref{ynulla} can be iterated in subsequent intervals of (at most) size~$t_0$. 
This iteration eventually leads to $E_{\rel}\equiv 0$ in $(0,T)$, concluding the proof. 
\end{proof}

\begin{rem}[Adding an external force]\label{R:ext_force}
Let us discuss how it is possible to add an external force $F$ to our analysis. Assume that the right-hand side of the first equation in~\eqref{euler} is equal to $F\in L^1([0,T];L^2(\R^d))$. In this case, note that the notion of \emph{admissibility} from \cref{admissible solution} has to be replaced with
\begin{equation*}
    \frac12  \|u(t)\|^2_{L^2(\R^d)}  \leq \frac12  \|u_0\|^2_{L^2(\R^d)}  + \int_0^t \int_{\R^d} F\cdot u\,dx\,ds
\end{equation*}
for $t\in[0,T]$.
Moreover, the energy balance \eqref{energy_balance_Ueps_prel} for $U_\varepsilon$ becomes
\begin{equation*}
E_{U_{\e}}(\tau)-E_{U_{\e}}(0)=-\int_0^\tau \int_{\R^d}R^U_{\e} :\nabla^s U_{\e} \,dx \, dt + \int_0^\tau \int_{\R^d} F_\e \cdot U_\e\,dx\,dt
\end{equation*}
and, consequently, \eqref{En_bal_U} becomes
\begin{equation*}
    \frac12  \|U(t)\|^2_{L^2(\R^d)} = \frac12 \|U(0)\|^2_{L^2(\R^d)}+ \int_0^t \int_{\R^d} F\cdot U\,dx\,ds,
\end{equation*}
for $t\in[0,T]$.
Furthermore, it is easy to see that, in this case, the extra term 
\begin{equation*}
E^\varepsilon_{F}(\tau):=\int_0^\tau \int_{\R^d} \left( F - F_\varepsilon \right) \cdot u \, dx\, dt + \int_0^\tau \int_{\R^d} \left( F_\e - F\right) \cdot U_\e \, dx\, dt 
\end{equation*}
appears as an additional error in the right-hand side of the estimate \eqref{est_E_rel_eps_fin}.  By the Dominated Convergence Theorem, we have $E^\varepsilon_F(\tau)\rightarrow 0$ uniformly in $[0,T]$. Thus, the very same inequality \eqref{dis_bella} still holds, and the proof of \cref{t:main1} does not modify.
\end{rem}

\section{Proof of \texorpdfstring{\cref{t:main2}}{convergence}}
\label{sec:proof_inviscid}

The proof of \cref{t:main2} is very similar to that of \cref{t:main1}. 
More precisely, given $\nu>0$, we want to estimate the \emph{viscous relative energy} 
\begin{equation}
    E_{\rel}^\nu(\tau) = \frac{1}{2} \int_{\R^d} \abs{U(\tau, x) - u^\nu(\tau,x)}^2 \, dx
\end{equation}
uniformly with respect to $\tau$. 
To begin, we prove the following result, which rephrases \cref{relative energy estimate} in the present context.

\begin{prop}[Key estimate on viscous relative energy] \label{relative energy estimate van visc}
Fix $\nu > 0$. Under the assumptions of \cref{t:main2}, we have
\begin{equation}
    E_{\rel}^\nu(\tau_2) \leq E_{\rel}^\nu(\tau_1) + \int_{\tau_1}^{\tau_2} \int_{\R^d} \nabla^s U : (U-u^\nu)\otimes (u^{\nu}-U) \, dx \, dt  + \nu \int_{\tau_1}^{\tau_2} \int_{\R^d} \nabla u^{\nu} \colon \nabla^s U \, dx \, dt \label{relative energy estimate bis}
\end{equation}
for every $\tau_2\in (0,T)$ and a.e.\  $\tau_1<\tau_2$, including $\tau_1=0$.
\end{prop}

\begin{proof} 
It is enough to prove \eqref{relative energy estimate bis} for a.e.\  $\tau_1,\tau_2\in (0,T)$. Indeed, since both~$u^\nu$ and~$U$ can be redefined on a negligible set of times in order to have $u^\nu,U\in C_t(w-L^2_x)$, having~\eqref{relative energy estimate bis} for a.e.\ time $\tau_2$ is equivalent to have the same inequality for every $\tau_2$ by lower semicontinuity of the $L^2_x$ norm under weak convergence.

As before, we would like to use $U$ as a test function in the weak formulation of $u^\nu$. To this aim, we consider the space mollification $U_\e$ of $U$. As shown in the proof of \cref{energyconsforU}, it is immediate to check that 
\begin{equation*}
U_\e \in \Lip([0,T]\times\R^d;\R^d)\cap\Lip([0,T]; W^{1,\infty}(\R^d;\R^d)\cap W^{1,2}(\R^d; \R^d)).
\end{equation*} 
Now, let us define 
\begin{equation}\nonumber
    E_{\rel}^{\nu,\e} (\tau) = \frac{1}{2} \int_{\R^d} \abs{u^\nu(x, \tau) - U_{\e}(x, \tau)}^2 \, dx \label{E-e-nu}
\end{equation}
for all $\tau\in[0,T]$ and $\e>0$. 
We have that 
\begin{equation}
    E_{\rel}^{\nu,\e} (\tau) = \frac{1}{2} \int_{\R^d} \abs{u^\nu(x, \tau)}^2 \, dx + \frac{1}{2} \int_{\R^d} \abs{U_\e(x, \tau)}^2 \, dx - \int_{\R^d} u^\nu(x,\tau) \cdot U_\e(x,\tau)\, dx. \label{E-e-nu 1}
\end{equation}
Since $U_\e$ is regular enough, we can introduce the usual commutator 
$$R_\e^U = U_\e \otimes U_\e - (U \otimes U)_\e $$
and recognize the validity of the following energy balance for $U_\e$,
\begin{align*}
    \frac{1}{2} \int_{\R^d} \abs{U_\e(x, \tau_2)}^2  \, dx 
    & = \frac{1}{2} \int_{\R^d} \abs{U_\e (x,\tau_1)}^2 \, dx - \int_{\tau_1}^{\tau_2} \int_{\R^d} R_{\e}^U \colon \nabla U_\e \, dx \, dt 
    \\ & = \frac{1}{2} \int_{\R^d} \abs{U_\e (x,\tau_1)}^2 \, dx - \int_{\tau_1}^{\tau_2} \int_{\R^d} R_{\e}^U \colon \nabla^s U_\e \, dx \, dt,
\end{align*}
where in the last equality we used that $R_\e^U$ is a symmetric matrix. Since $u^\nu$ is a Leray--Hopf weak solution of the Navier--Stokes system~\eqref{navierstokes}, we have that 
\begin{equation}
    \int_{\R^d} \abs{u^\nu(x,\tau_2)}^2 \ dx \leq \int_{\R^d} \abs{u^\nu(x,\tau_1)}^2 \ dx \nonumber
\end{equation}
for every $\tau_2\in (0,T)$ and a.e.\ $\tau_1<\tau_2$, including $\tau_1=0$. Thus, with the same computations shown in details in the proof of \cref{relative energy estimate}, by \eqref{E-e-nu 1} we obtain that  
\begin{align}
    E_{\rel}^{\nu,\e} (\tau_2) & \leq  E_{\rel}^{\nu,\e} (\tau_1) - \int_{\tau_1}^{\tau_2} \int_{\R^d} R_\e^U \colon \nabla^s U_\e \, dx\,  dt - \int_{\tau_1}^{\tau_2} \int_{\R^d}u\cdot \divergence{R_\e^U} \, dx \, dt \nonumber
    \\ & \hspace{0.4 cm} + \int_{\tau_1}^{\tau_2} \int_{\R^d} \nabla^s U_\e \colon (U_\e-u^\nu) \otimes (u^\nu - U_\e )\, dx \, dt + \nu \int_{\tau_1}^{\tau_2} \int_{\R^d} \nabla u^\nu \colon \nabla^s U_\e \, dx \, dt \nonumber
\end{align}
for a.e.\ $\tau_2\in (0,T)$ and a.e.\ $\tau_1<\tau_2$, including $\tau_1=0$. Letting $\e\to 0^+$ as in \cref{relative energy estimate} and using \cref{third term}, we get \eqref{relative energy estimate bis}, concluding the proof.
\end{proof}

\begin{rem}[Uniform bound on the viscous relative energy for small times] \label{uniform in nu bound}
As a product of \cref{relative energy estimate van visc}, we can obtain a  bound (uniform in $\nu$) of the relative energy for small times, which will be crucial in the proof of \cref{t:main2}. 
Precisely, under the assumptions of \cref{t:main2}, from inequality~\eqref{relative energy estimate bis} we get that
\begin{align*}
    E_{\rel}^\nu (\tau+\tau_1) & \leq E_{\rel}^\nu (\tau_1)+  \int_{\tau_1}^{\tau+\tau_1} \norm{\nabla^s U(t)}_{L^{(1+\frac{\sigma}{2})'}_x} \norm{ U(t) - u^\nu(t)}_{L_x^{2+\sigma}}^2 \, dt  
    \\ & \quad +\nu \int_{\tau_1}^{\tau+\tau_1} \norm{\nabla u^\nu(t)}_{L^2_x} \norm{\nabla^s U(t)}_{L^2_x} \, dt  
    \\ &\leq E_{\rel}^\nu (\tau_1) +C(\sigma)  \int_{\tau_1}^{\tau+\tau_1} \left( f(t)\, h(t)^2 + g(t)^2 \right) \, dt.
\end{align*}
for any $\tau>0$ and for a.e.\  $\tau_1\in [0,T]$, including $\tau_1=0$, and for any $\nu\in (0,1)$.
Since $f h^2 + g^2 \in L^1_t$, we deduce that, for any given $\eta>0$, we can find a time $\tau_\eta\in(0,T]$, depending on~$\eta$ only, such that 
\begin{equation}
     E_{\rel}^\nu (\tau+\tau_1)  \leq E_{\rel}^\nu (\tau_1)+ \eta
     \label{uniform in nu bound estimate}
\end{equation}
for all $\tau\in [0,\tau_\eta]$ and a.e.\  $\tau_1\in [0,T]$, including $\tau_1=0$.
\end{rem}

We can now detail the proof of \cref{t:main2}. 
The strategy is very similar to that followed in the proof of \cref{t:main1}, since we just need to combine  \cref{relative energy estimate van visc} together with a (integral) Gr\"onwall-type argument. 

\begin{proof} [Proof of \cref{t:main2}] We divide the proof in three steps.
\smallskip 

\textit{Step~1: integral inequality for the viscous relative energy}.
Since $E_{\rel}^\nu(0)=0$ by assumption, by \cref{relative energy estimate van visc} we can estimate
\begin{equation}
    E_{\rel}^\nu(\tau) \leq  \int_0^\tau \int_{\R^d} \nabla^s U: (U-u^\nu)\otimes ( u^\nu-U) \, dx \, dt  + \nu \int_0^\tau \int_{\R^d} \nabla u^{\nu} \colon \nabla^s U \, dx \, dt 
    \label{starting point bis}
\end{equation}
for all $\tau \in (0, T)$.
As in the proof of \cref{t:main1}, we consider the quantity 
\begin{equation*}
\alpha_\nu(x,t) = \abs*{U(x,t)-u^\nu(x,t)}^2,
\qquad
(t,x)\in[0,T]\times\R^d.
\end{equation*}
Letting $\gamma \in (0, 1]$ to be chosen later, we set
\begin{equation}
\label{alpha_rotto-nu}
\alpha_\nu^+ = \alpha\, \mathbf{1}_{\set*{\alpha > \nu^{-\gamma}}},
\qquad
\alpha_\nu^- = \alpha\, \mathbf{1}_{\set*{\alpha \leq \nu^{-\gamma} }},
\end{equation}
for any $\nu \in (0,1)$.
Thanks to~\eqref{starting point bis} and \cref{orlicz-holder}, we can estimate
\begin{equation}
\label{alpha-nu stima fino a g}
\begin{split}
      \norm{\alpha_\nu(\tau)}_{L^1_x}  &\leq  2 \int_0^\tau \int_{\R^d} \alpha_\nu(t)\, \abs{\nabla^s U(t)} \, dx \, dt  + 2 \nu \int_0^\tau \norm{\nabla u^\nu(t)}_{L^2_x}\,\norm{\nabla^s U(t)}_{L^2_x} \, dt 
    \\ &  \leq  2 \int_0^\tau \int_{\R^d} \alpha_\nu^-(t) \,\abs{\nabla^s U(t)} \, dx \, dt + 2 \int_0^\tau \int_{\R^d} \alpha_\nu^+(t)\, \abs{\nabla^s U(t)}\, dx \, dt 
    \\ & \quad +  2\nu \int_0^\tau \norm{\nabla u^\nu(t)}_{L^2_x}\, \norm{\nabla^s U(t)}_{L^2_x} \, dt 
    \\ & \leq   C \int_0^\tau \norm{\nabla^s U(t)}_{L^{\exp}_x}\, \norm{\alpha_\nu^-(t)}_{L^1_x} \bigg[ \left| \log(\norm{\alpha_\nu^-(t)}_{L^1_x})\right| + \log(\norm{\alpha_\nu^-(t)}_{L^\infty_x} + 1) +1 \bigg] \,dt
    \\ & \quad + 2 \int_0^\tau \int_{\R^d} \alpha_\nu^+(t)\, \abs{\nabla^s U(t)} \, dx \, dt  + 2 \nu \int_0^\tau \norm{\nabla u^\nu(t)}_{L^2_x}\, \norm{\nabla^s U(t)}_{L_x^2} \, dt 
    \\ &  \leq  C \int_0^\tau f(t)\, \norm{\alpha_\nu(t)}_{L_x^1} \left[ \left| \log(\norm{\alpha_\nu(t)}_{L_x^1})\right| + \log\left(\nu^{ - \gamma} + 1\right) +1 \right] \, dt 
    \\ & \quad + C \int_{0}^\tau \left( g_0(t) + \sqrt \nu\, g(t)^2\right) \, dt
\end{split}    
\end{equation}
for all $\nu \in (0,1)$ and all $\tau \in [0, T]$, where we defined
$$g_0(t) = \int_{\R^d} \alpha_\nu^+(x,t)\, \abs{\nabla^s U(x,t)} \, dx = \int_{\{\alpha_\nu(\cdot,t) > \nu^{-\gamma}\}} \alpha_\nu(x,t)\, \abs{\nabla^s U(x,t)} \, dx .$$
Note that, in the last step of the above chain of inequalities, we used the fact that the function 
$$\psi(s) = s\,\big[ \abs{ \log(s)} + \log(1+ \nu^{-\gamma}) + 1\big]$$ 
is non-decreasing for any $\nu \in (0,1)$. 
We now need to estimate the function $g_0$. 
To this aim, since $\nabla^s U (\cdot,t) \in L^{\exp}_x$, we note that $\nabla^s U (\cdot,t) \in L^{(1+\sigma/4)'}_x$, so there exists a constant $C(\sigma)>0$ such that 
$$\norm{\nabla^s U (t)}_{L^{(1+\sigma/4)'}_x} \leq C(\sigma) \norm{\nabla^s U(t) }_{L^{\exp}_x} \leq C(\sigma) f(t), \quad f \in L^1_t.  $$
Recalling that $\alpha_\nu(t) = \abs{U(t)-u^\nu(t)}^2 \in L^\infty_t(L^{1+\sigma/2}_x)$, by Chebyshev's inequality we get that
\begin{equation*}
    \left| \{\alpha_\nu(t) > \nu^{-\gamma}\} \right| \leq \left(\nu^\gamma \norm{ \alpha_\nu(t) }_{L_x^{(1+\sigma/2)}}  \right)^{1+\frac{\sigma}{2}}. 
\end{equation*}
Thus, by Holder's inequality, we infer that
\begin{align*}
    g_0(t) & \leq \norm{\nabla^s U(t)}_{L^{(1+\sigma/4)'}_x} \,\norm{\alpha_\nu(t)^+_\nu}_{L^{(1+\sigma/4)}_x} 
    \\ & \leq C(\sigma) \,f(t)\, \abs{\{\alpha_\nu(t)> \nu^{-\gamma}\} }^{\frac{2\sigma}{(2+\sigma)(4+\sigma)}} \,\norm{\alpha_\nu(t)}_{L^{1+\sigma/2}_x}
    \\ & \leq C(\sigma)\, f(t)\, \nu^{\frac{\gamma \sigma}{4+\sigma}}\, \norm{\alpha_\nu(t)}_{L^{1+\sigma/2}_x}^{1+\frac{\sigma}{4+\sigma}}
    \\ & \leq C(\sigma) \,f(t)\, h(t)^{1+\frac{\sigma}{\sigma +4}}\, \nu^{\frac{\gamma\sigma}{4+\sigma}}.
\end{align*} 
Therefore, going back to~\eqref{alpha-nu stima fino a g}, we have that 
\begin{equation}
\label{def funzioni p e q}
\begin{split}
    \norm{\alpha_\nu(\tau)}_{L^1_x} 
    &\leq   C \int_0^\tau f(t)\, \norm{\alpha_\nu(t)}_{L^1_x} \left[ \left| \log(\norm{\alpha_\nu(t)}_{L^1_x})\right| + \log\left(\nu^{-\gamma} + 1\right) +1 \right] \, dt  
    \\ & \quad + C \int_0^\tau \left( f(t)\, h(t)^{1+\frac{\sigma}{\sigma+4}} \,\nu^{\frac{\gamma\sigma}{\sigma + 4}} + \sqrt \nu \,g(t)^2 \right) \, dt 
    \\ & \leq   \nu^{\frac{\gamma\sigma}{\sigma + 4}} \int_0^\tau q(t)\, dt 
   + \int_0^\tau p(t)\, \norm{\alpha(t)}_{L^1_x} \left( \left|\log(\norm{\alpha(t)}_{L^1_x})\right| + \log(\nu^{-\gamma} + 1) +1 \right) \, dt,  
\end{split}    
\end{equation}
for all $\tau \in [0, T]$ and all $\nu\in (0, 1)$.
Here and in what follows, $p, q \in L^1([0,T])$ are two non-negative functions depending on $f$, $g$, $h$ and~$\sigma$ only.
Moreover, we assumed that $\sigma$ is small enough to guarantee that $\frac{\sigma}{\sigma +4 }\le\frac12$. 
Note that this assumption is not restrictive, in virtue of the interpolation between $L^2_x$ and $L_x^{2+\sigma}$.

\smallskip

\textit{Step~2:  Gr\"onwall's inequality}.
We now choose $\gamma = 1$. 
Letting
\begin{equation*}
y_\nu(\tau)= \norm{\alpha_\nu(\tau)}_{L_x^1},
\qquad
\theta=\frac{\sigma}{4+\sigma},
\qquad
C_\nu=(\nu^{-1} + 1) e,
\end{equation*}
we can equivalently rewrite inequality~\eqref{def funzioni p e q} as
\begin{equation}\label{casino1}
y_\nu(\tau) \leq    \nu^{\theta}\int_0^\tau q(t)\,dt + \int_0^\tau p(t) y_\nu(t)\left(\left|\log y_\nu (t)\right|+\log C_\nu \right) \, dt.
\end{equation}
Define $\tilde y_\nu (\tau)=y_\nu (\tau)+\nu $. Since the function $y\mapsto y\left( \left| \log y\right|+\log C_\nu \right)$ is increasing, we get 
$$
\tilde y_\nu (\tau)\leq \nu +\nu^\theta\int_0^\tau q(t)\,dt +\int_0^\tau p(t)\tilde  y_\nu (t) \left(\left|\log \tilde y_\nu (t)\right|+\log C_\nu \right) \, dt.
$$
In view of \cref{uniform in nu bound} we can pick a time $t_0>0$ such that
$y_\nu (\tau)\leq \frac12$
for every $\nu \in (0,1) $ and for every $\tau\in [0,t_0]$.
In particular, for $\nu\leq \frac12$, we have
$$
\nu\leq \tilde y_\nu(\tau)\leq 1
\quad 
\text{for all}\ \tau\in [0,t_0].
$$
We remark that the choice of $t_0$ is independent of $\nu$. 
With this choice, we have $|\log \tilde y_\nu (\tau)|\leq -\log \nu $ and thus
$$
\tilde y_\nu (\tau)\leq \nu^\theta\left(1+\int_0^\tau q(t)\,dt\right) + (\log M_\nu)  \int_0^\tau   p(t)\,\tilde  y_\nu (t) \, dt,
$$
where we defined $M_\nu =C_\nu/\nu$. Note that 
\begin{equation}
    \label{boundsuM}
    M_\nu =\frac{(\nu^{-1}+1)e}{\nu }\leq \frac{C}{\nu^2},
\end{equation}
for some universal constant $C>0$. By Gr\"onwall's inequality, we get that
\begin{equation*}
    \tilde y_\nu(\tau)\leq \nu^\theta\left(1+\int_0^\tau q(t)\,dt\right)e^{ \log M_\nu  \int_0^\tau p(t)\,dt }\leq C\nu^{\theta-2\int_0^\tau p(t)\,dt},
\end{equation*}
for all $\tau\in [0,t_0]$,
where in the last inequality follows from \eqref{boundsuM} and from the fact $q\in L^1([0,T])$. 
Since also $p\in L^1([0,T])$, by possibly decreasing further $t_0$ depending on $\theta$ and $p$ but not on $\nu$, we can enforce
\begin{equation}
    \label{int_p_piccolo}
    \sup_{s\in[0,T-t_0]}\int_{s}^{s+t_0}p(t)\,dt\leq \frac{\theta}{4},
\end{equation}
yielding that $\tilde y_\nu (\tau)\leq C\nu ^\frac{\theta}{2}$ for all $\tau\in[0,t_0]$ for some  constant $C>0$ only depending on $\theta$ and on the functions~$f$, $g$ and~$h$ in the statement of \cref{t:main2}. 
In particular, we have
$$
y_\nu (\tau)\leq C\nu ^\frac{\theta}{2}
$$
for all $\tau\in [0,t_0]$ and all $\nu \in (0, 1/2)$.

\smallskip 

\textit{Step~3: iteration argument}.
Now we need to iterate the same procedure followed in Step~2 in every interval of the form $[it_0,(i+1)t_0]$, $i\in\mathbb N$, until we cover the whole interval $[0,T]$ in a finite number of steps, say $N\in\mathbb N$, where $t_0>0$ is the positive time chosen in Step~2. We remark that $t_0$ depends only on $\sigma$ and on the equi-integrability of the functions $f$ and $f h^2 + g^2$, recall \cref{uniform in nu bound}. 

To fix the ideas, we describe in details how to get the desired estimate in the interval $[t_0,2t_0]$. 
Since $2E^\nu_{rel}(t_0)=y_\nu(t_0)\leq C\nu^{\frac{\theta}{2}}$ according to Step~2, by \cref{relative energy estimate van visc}  together with \cref{uniform in nu bound}, we have 
$$
y_\nu(\tau)\leq C\nu^{\frac{\theta}{2}}+\frac12
$$ 
for all $\tau\in[t_0,2t_0]$.
Now define $\tilde y_\nu(\tau)=y_\nu(\tau)+\nu^{\frac{\theta}{2}}$. It is clear that, by possibly choosing~$\nu$ even smaller depending on $C$ and $\theta$ only, we can ensure that 
$$
\nu^{\frac{\theta}{2}}\leq \tilde y_\nu(\tau)\leq 1
$$
for all $\tau\in[t_0,2t_0]$.
By repeating all the computations from the beginning (that is, by considering~\eqref{def funzioni p e q} starting from $\tau=t_0$ instead of $\tau=0$) and this time choosing $\gamma = \frac{\theta}{2}$, we get
\begin{align*}
\tilde y_\nu (\tau)&\leq 2 E_{\rel}^\nu (t_0)+\nu^{\frac{\theta}{2}}+ \nu^{\frac{\theta^2}{2}}\left(1+\int_{t_0}^\tau q(t)\,dt\right) 
\\
&\quad+\int_{t_0}^\tau p(t)\,\tilde  y_\nu (t)\left( \log \left(1+\tfrac{1}{\nu^{\theta/2}}\right)+\log\left(\tfrac{1}{\nu^{\theta/2}}\right) \right) dt\\
&\leq C\nu^{\frac{\theta^2}{2}}\left(1+\int_{t_0}^\tau q(t)\,dt\right)+C \int_{t_0}^\tau p(t)\,\tilde  y_\nu (t)\log\left(\tfrac{1}{\nu^\theta}\right) \, dt
\end{align*}
for all $\tau\in[t_0,2t_0]$.
Therefore, once again by Gr\"onwall's inequality, we deduce that
$$
\tilde y_\nu (\tau)\leq C^2 \nu^{\frac{\theta^2}{2}}\left(1+\int_{t_0}^\tau q(t)\,dt\right)\left(\frac{1}{\nu^\theta}\right)^{\int_{t_0}^{\tau}p(t)\,dt}\leq C^2 \nu^{\frac{\theta^2}{2}-\theta\int_{t_0}^{\tau}p(t)\,dt}
$$
for all $\tau\in[t_0,2t_0]$.
By our choice of $t_0$ in~\eqref{int_p_piccolo}, we have $\int_{t_0}^{2t_0}p(t)\,dt\leq \frac{\theta}{4}$, from which
$$
\nu^{\frac{\theta^2}{2}-\theta\int_{t_0}^{2t_0}p(t)\,dt}\leq \nu^{\frac{\theta^2}{2}-\frac{\theta^2}{4}}\leq \nu ^{\frac{\theta^2}{4}}.
$$
Thus 
$$
\tilde y_\nu (\tau)\leq C^2 \nu ^{\frac{\theta^2}{4}}
$$
for all $\tau\in[t_0,2t_0]$.
Repeating the above argument for the subsequent time intervals of size $t_0$ for a finite number of steps, say $N$, we can cover the whole interval $[0,T]$ and we finally obtain
$$
y_\nu(\tau)\leq C^N\nu^{\left(\frac{\theta}{2}\right)^N}
$$
for all $\tau\in[0,T]$, where $C>0$ is a constant only depending on $\sigma$ and the functions $f$, $g$ and $h$ in the statement of \cref{t:main2}.
The desired estimate in~\eqref{rate 1} thus follows by setting $M=\max\left\{C^N, \left(\frac{2}{\theta}\right)^N\right\}$ and the proof is complete.
\end{proof}


\end{document}